\documentclass{amsart} 
\pdfoutput=1	
\title{Explicit arithmetic intersection theory and computation of N\'eron-Tate heights}

\usepackage{amsmath, amssymb, mathrsfs, amsthm, shorttoc, geometry, stmaryrd, comment}
\usepackage{mathtools} 
\usepackage{colonequals} 
\usepackage{xcolor}
\definecolor{darkgreen}{rgb}{0,0.5,0}
\usepackage[
        colorlinks, citecolor=darkgreen,
        backref,
        pdfauthor={Raymond van Bommel, David Holmes, J. Steffen M\"uller}
        pdftitle={Explicit arithmetic intersection theory and computation of N\'eron-Tate heights}
]{hyperref}

\usepackage[alphabetic, backrefs]{amsrefs} 
\usepackage[all]{xy}

\usepackage{enumitem}

\usepackage{cleveref}
\let\oref\ref
\AtBeginDocument{\renewcommand{\ref}[1]{\cref{#1}}}

\usepackage{pgf,tikz}
\usetikzlibrary{matrix, calc, arrows}

\usepackage{tikz-cd} 

\makeatletter
\newcommand*{\doublerightarrow}[2]{\mathrel{
  \settowidth{\@tempdima}{$\scriptstyle#1$}
  \settowidth{\@tempdimb}{$\scriptstyle#2$}
  \ifdim\@tempdimb>\@tempdima \@tempdima=\@tempdimb\fi
  \mathop{\vcenter{
    \offinterlineskip\ialign{\hbox to\dimexpr\@tempdima+1em{##}\cr
    \rightarrowfill\cr\noalign{\kern.5ex}
    \rightarrowfill\cr}}}\limits^{\!#1}_{\!#2}}}
\newcommand*{\triplerightarrow}[1]{\mathrel{
  \settowidth{\@tempdima}{$\scriptstyle#1$}
  \mathop{\vcenter{
    \offinterlineskip\ialign{\hbox to\dimexpr\@tempdima+1em{##}\cr
    \rightarrowfill\cr\noalign{\kern.5ex}
    \rightarrowfill\cr\noalign{\kern.5ex}
    \rightarrowfill\cr}}}\limits^{\!#1}}}
\makeatother

\newcommand{\on}[1]{\operatorname{#1}}
\newcommand{\bb}[1]{{\mathbb{#1}}}

\newcommand{\ca}[1]{{\mathcal{#1}}}

\newcommand{\Span}[1]{\left<#1\right>}

\newcommand{\sub}{\subseteq}

\newcommand\Q{\mathbb{Q}}
\newcommand\C{\mathbb{C}}
\newcommand\F{\mathbb{F}}

\newcommand\Z{\mathbb{Z}}
\newcommand\R{\mathbb{R}}

\newcommand\Gal{\mathop{\rm Gal}\nolimits}

\newcommand\supp{\mathop{\rm supp}\nolimits}

\newcommand\tors{\mathop{\rm tors}\nolimits}

\newcommand{\Div}{\operatorname{Div}}

\renewcommand{\div}{\operatorname{div}}
\newcommand{\GL}{\operatorname{GL}}

\theoremstyle{definition}
\newtheorem{definition}{Definition}[section]

\newtheorem{situation}[definition]{Situation}

\newtheorem{algorithm}[definition]{Algorithm}

\theoremstyle{plain}

\newtheorem{proposition}[definition]{Proposition}
\newtheorem{lemma}[definition]{Lemma}
\newtheorem{theorem}[definition]{Theorem}

\theoremstyle{remark}

\newtheorem{remark}[definition]{Remark}

\renewcommand{\phi}{\varphi}

\author{Raymond van Bommel}
\address{Raymond van Bommel, Mathematisch Instituut, 
Universiteit Leiden,
Postbus 9512,
2300 RA Leiden,
Netherlands}

\author{David Holmes}
\address{David Holmes,
Mathematisch Instituut,
Universiteit Leiden,
Postbus 9512,
2300 RA Leiden,
Netherlands}

\author{J. Steffen M\"uller}
\address{J. Steffen M\"uller,
  Bernoulli Institute, 
  University of Groningen,
  Nijenborgh 9,
  9747 AG Groningen,
  Netherlands
}

\date{\today}
\newcounter{nootje}
\newcommand{\expanded}[1]{#1}
\renewcommand{\expanded}[1]{}

\newcommand{\beq}{\begin{equation}}
\newcommand{\eeq}{\end{equation}}
\newcommand{\beqs}{\begin{equation*}}
\newcommand{\eeqs}{\end{equation*}}

\renewcommand{\Im}{\mathop{\mathrm{Im}}}
\renewcommand{\Re}{\mathop{\mathrm{Re}}}

\begin{document}

\begin{abstract} 
We describe a general algorithm for computing intersection pairings on arithmetic
  surfaces. We have implemented our algorithm for curves over $\bb Q$, and we show how to
  use it to compute regulators for a number of Jacobians of smooth plane quartics, and to
  numerically verify the conjecture of Birch and Swinnerton-Dyer for the Jacobian of the
  split Cartan curve of level~13, up to squares.
\end{abstract}

\maketitle

\newcommand{\Mtildes}{ \widetilde{\ca M}^\Sigma}
\newcommand{\sch}[1]{\textcolor{blue}{#1}}

\section{Introduction}
If $A/K$ is an abelian variety over a global field $K$, then an ample symmetric divisor
class $c$ on $A$ induces a non-degenerate quadratic form $\hat{h}_c$ on $A(K)$, the {\em N\'eron-Tate
height} or {\em canonical height} with respect to $c$. Given $P \in A(K)$, the height of $P$ can be defined as 
\begin{equation*}
\hat{h}_c(P) = \lim_{n \to \infty} \frac{1}{n^{2}}h_c(nP),
\end{equation*}
where $h_c$ is a Weil height on $A$ induced by $c$ (see~\cite{Ner65} and \cite[Section B.5]{HS00}). 
The N\'eron-Tate height also induces a symmetric bilinear pairing on $A(K)$ given by $$\hat{h}_c(P,Q) = \frac12\left(\hat{h}_c(P+Q) - \hat{h}_c(P) - \hat{h}_c(Q)\right).$$

An algorithm to compute the N\'eron-Tate height is required, for instance, to compute
generators of $A(K)$.
 More precisely, the canonical height endows 
 $A(K)\otimes_K \R$ with the structure of a Euclidean vector space and $A(K)/A(K)_{\tors}$
 embeds into this vector space as a lattice $\Lambda$. Given generators of a subgroup of $A(K)/A(K)_{\tors}$ of
 finite index, we can find generators of the full group by saturating the corresponding
 sublattice of $\Lambda$. All known methods for this saturation step require an algorithm to compute the canonical height 
 (see~\cite{Sik95, FS97, Sto02}).  
 Another important application is the computation of the regulator of $A/K$, a quantity which appears in the conjecture of
 Birch and Swinnerton-Dyer. 
 The regulator of $A/K$ is the Gram determinant of a set of generators of $\Lambda$ (for a certain choice of
 $c$). If we only have generators of a finite index subgroup available, then we can still
 compute the regulator up to an integral square factor.

We can construct $\hat{h}_c$ explicitly if we have explicit formulas for a map to projective space corresponding to
the linear system of $c$. For instance, an explicit embedding of the Kummer
variety of $A$ has been used to give algorithms for the computation of N\'eron-Tate
heights for elliptic curves~\cite{Sil88, MS16a} and Jacobians of hyperelliptic curves of genus
2~\cite{FS97, Sto02, MS16b} and genus 3~\cite{Sto17}.
However, this approach becomes quickly infeasible if we increase the dimension of $A$.

But if $J$ is the Jacobian variety of a smooth projective geometrically connected curve $C/K$, then there is an
alternative way due to Faltings and Hriljac to describe the N\'eron-Tate height on $J/K$ with respect to twice 
the class of a symmetric theta divisor as follows (see \ref{S:Gross} for details):
\begin{equation}\label{eq:FH}
  \hat h_{2 \vartheta}([D],[E]) = -\sum_{v \in M_K} \Span{D,E}_v. 
\end{equation}
Here $D$ and $E$ are  two divisors of degree $0$ on $C$ without common component, 
$M_K$ denotes the set of places of $K$, and $\Span{D,E}_v$ denotes the local N\'eron
pairing of $D$ and $E$ at $v$, which is defined below in sections \oref{sec:non_arch} (for the non-archimedean places) and \oref{sec:arch} (for the archimedean places). 

In this note, we show how to turn \ref{eq:FH} into an algorithm for
computing $\hat{h}_{2\vartheta}$ when $K=\Q$ (our algorithm can be generalised easily to work
over general global fields).
This was already done independently by the second-named and the third-named authors
in~\cite{Hol12} and~\cite{Mul14} in
the special case of hyperelliptic curves. 
But for Jacobians of non-hyperelliptic curves, no practical algorithms for
computing N\'eron-Tate heights are known, and therefore no numerical evidence for the Birch and
Swinnerton-Dyer conjecture has been collected.

In the present paper we develop such an algorithm and we 
give numerical evidence for the conjecture of Birch and
Swinnerton-Dyer for a number of Jacobians, including that of the split Cartan modular curve of level 13.
Our main contribution is a new way to compute the non-archimedean
local N\'eron pairings. In fact, we give a new algorithm for computing the intersection
pairing of two divisors without common component on a regular arithmetic surface, which might
be of independent interest.
In short, we lift divisors from the generic fibre to the arithmetic surface by saturating the
defining ideals, and we use an inclusion-exclusion principle to deal with divisors
intersecting on several affine patches. 
The archimedean local N\'eron pairings $\Span{D,E}_\infty$ are computed in essentially the
same way as in in~\cite{Hol12} and~\cite{Mul14}, by pulling back a translate of the
Riemann theta function to $C(\C)$. This requires explicitly computing period matrices and
Abel-Jacobi maps on Riemann surfaces; we use the recent algorithms of Neurohr~\cite[Chapter
4]{NeurohrPhD} and Molin-Neurohr~\cite{MN19}.

The paper is organised as follows: In~\ref{sec:non_arch} we introduce our algorithm to compute non-archimedean
local N\'eron pairings. The computation of archimedean local N\'eron pairings is discussed in
\ref{sec:arch}. The topic of \ref{sec:global} is how to apply these to compute canonical
heights using \ref{eq:FH}. Finally, in~\ref{sec:examples} we demonstrate the practicality of our
algorithm by computing the N\'eron-Tate regulator, up to an integral square, for several
Jacobians of smooth plane quartics including the split (or, equivalently, non-split)
Cartan modular curve of level 13, and we numerically verify BSD for the latter curve up to an
integral square.

\subsection{Acknowledgements}
Most of the work for this paper was done when the authors were participating in the
workshop ``Arithmetic of curves'', held in Baskerville Hall in August~2018. We would like
to thank the organisers Alexander Betts, Tim Dokchitser, Vladimir Dokchitser and C\'eline
Maistret, as well as the Baskerville Hall staff, for providing a great opportunity to concentrate on this project.
We also thank Christian Neurohr for sharing his code to compute Abel-Jacobi maps for
general curves and for answering several questions, and Martin Bright for suggesting the
use of the saturation. 
Finally, we are very grateful to the anonymous referee for a thorough and rapid report. 

\section{The non-archimedean N\'eron pairing}\label{sec:non_arch}
\newcommand{\fp}{\mathfrak p}
\renewcommand{\fp}{p}
\newcommand{\Qp}{\bb Q_{\fp}}
\newcommand{\Zp}{\bb Z_{\fp}}
For simplicity of exposition, we restrict ourselves to curves over the rational numbers;
everything we do generalises without substantial difficulty to global fields. 
For background on arithmetic surfaces and their intersection pairing, we refer to Liu's book
\cite{Liu02}. In this section we work over a fixed prime $\fp$ of $\bb
Z$. Let $C/\Qp$ be a smooth proper geometrically connected curve, and let $\ca C/\Zp$ be a proper
regular model of $C$. Because $\ca C$ is a regular surface, we have an intersection pairing between
divisors on $\ca C$ having no components in common; if $\mathcal{P}$ and $\mathcal{Q}$ are distinct prime divisors the pairing is given by 
\begin{equation*}
  \iota(\mathcal{P} , \mathcal{Q}) = \sum_{P \in  \ca C^0} \on{length}_{\ca O_{\ca C,
  P}}\left(\frac{\ca O_{\ca C, P}}{\ca O_{\ca C, P}(-\mathcal{P}) + \ca O_{\ca C,
  P}(-\mathcal{Q})}\right)\log \# k(P); 
\end{equation*}
here $\ca C^0$ denotes the set of closed points of $\ca C$, and $k(P)$ denotes the residue
field of the point $P$. We extend to arbitrary divisors with no common components by additivity. 

In general, this intersection pairing fails to respect linear equivalence. However, if
$\ca D$ is a divisor on $\ca C$ whose restriction to the generic fibre $C$ has degree $0$,
and $Y$ is a divisor on $\ca C$ pulled back from a divisor on $ \on{Spec} \Zp$, then $\ca
D \cdot Y = 0$. By the usual formalism with a moving lemma, this allows us to define the intersection pairing between any two divisors $\ca D$ and $\ca E$ on $\ca C$ as long as the restrictions of $\ca D$ and $\ca E$ to the generic fibre $C$ have degree $0$ and disjoint support. 

If $D$ is a divisor on $C$, we write $\ca D$ for the unique horizontal divisor on $\ca C$
whose generic fibre is $D$. For a divisor $D$ of degree $0$ on $C$, we write $\Phi(D)$ for
a vertical $\Q$-divisor on $\ca C$ such that for every vertical divisor $Y$ on $\ca C$, we have $\iota(Y , \ca D + \Phi(D)) = 0$; this $\Phi(D)$ always exists, and is unique up to the addition of divisors pulled back from $ \on{Spec} \Zp$ (see~\cite[Theorem~III.3.6]{Lan88}). 

Let $D$ and $E$ be two divisors on $C$, of degree $0$ and with disjoint support. Then the \emph{local N\'eron pairing between $D$ and $E$} is given by
\begin{equation*}
\Span{D,E}_\fp \coloneqq \iota(\ca D +\Phi(D),\ca E +\Phi(E)). 
\end{equation*}
This pairing is bilinear and symmetric, but it does not respect linear equivalence;
see~\cite[Theorem~III.5.2]{Lan88}.

Our goal in this section is to compute the pairing $\Span{D,E}_\fp$, assuming that $D$ and $E$ are given to us (arranging suitable $D$ and $E$, and identifying those primes $\fp$ which may yield a non-zero pairing, will be discussed in \ref{sec:global}). A first step in applying the above definitions is to compute a regular model of $C$ over $\bb Z_\fp$. Algorithms are available for this in {\tt Magma}, one due to Steve Donnelly, and another to Tim Dokchitser \cite{Dok18}. For our examples below we used Donnelly's implementation as slightly more functionality was available, but our emphasis in this section is on providing a general-purpose algorithm which should be easily adapted to take advantage of future developments in the computation of regular models. 

\subsection{The naive intersection pairing}
To facilitate the computation of the local N\'eron pairing at non-archimedean places, we will introduce a \emph{naive} intersection pairing, which coincides with the standard intersection pairing on regular schemes, and then give an algorithm to compute the naive intersection pairing in a fairly general setting. 

\begin{situation}\label{sit:affine}
We fix the following data:
\begin{itemize}
\item
An integral domain $R$ of dimension 2, flat and finitely presented over $\bb Z$;
\item 
effective Weil divisors $\ca D$ and $\ca E$ on $\ca C\coloneqq \on{Spec}R$ with no common irreducible component in their support, defined by the vanishing of ideals $I_{\ca D}$ and $I_{\ca E} $ in $ R$ (i.e. $I_{\ca D} = \ca O_{\ca C}(-\ca D) \sub \ca O_{\ca C}$, and analogously for $\ca E$); 
\item
a constructible subset $V$ of $\ca C$. 
\end{itemize}
\end{situation}
For computational purposes, we suppose that a finite presentation of $R$ is given, along with generators of $I_{\ca D}$ and $I_{\ca E}$. Moreover, we suppose that $V$ is given as a disjoint union of intersections of open and closed subsets. 
\begin{definition}
Let $P$ be a closed point of $\ca C$ lying over $\fp$. The \emph{naive intersection number of $\ca D$ and $\ca E$ at $P$} is given by
\begin{equation*}
\iota^{naive}_P(\ca D,  \ca E) \coloneqq \on{length}_{\ca O_{\ca C, P}}\left( \frac{\ca O_{\ca C, P}}{I_{\ca D,P} + I_{\ca E,P}}\right)\log \# k(P), 
\end{equation*}
  where $I_{D,p} = I_D\otimes {\ca O_{\ca C,P}}$ and likewise for $E$.
If $W$ is any subset of $\ca C$, we define
\begin{equation*}
\iota^{naive}_W(\ca D,  \ca E) \coloneqq \sum_{P \in W^0} \iota^{naive}_P(\ca D,  \ca E), 
\end{equation*}
where $W^0$ denotes the set of closed points in $W$ lying over $\fp$. 
\end{definition}

Note that if $\ca C$ is regular at $P$, then $\iota^{naive}_P(\ca D,  \ca E)$ is the usual intersection pairing $\iota_P(\ca D,  \ca E)$ at $P$. If $W$ and $W'$ are disjoint subsets of $C$, then 
\begin{equation}\iota^{naive}_W(\ca D,  \ca E) + \iota^{naive}_{W'}(\ca D,  \ca E) = \iota^{naive}_{W\cup W'}(\ca D,  \ca E). \label{lem:additivity}
\end{equation} 

We present here an algorithm for computing the naive intersection pairing $\iota^{naive}_V(\ca D,  \ca E)$ for $V$ any constructible subset of $\ca C$. This seems to us a reasonable level of generality to work in; constructible subsets are the most general subsets easily described by a finite amount of data, and should be flexible enough for computing local N\'eron pairings for any reasonable way a regular model is given to us. Note that only being able to compute the intersection pairing at points would not be sufficient, as we would then need to sum over infinitely many points, and only being able to compute it for $V$ affine gives complications where patches of the model overlap. 

\begin{algorithm}\label{alg:NA}
Suppose we are in \ref{sit:affine}. The following is an algorithm to compute $\iota^{naive}_V(\ca D,  \ca E)$. 
\end{algorithm}

\textbf{First reduction step:} By \ref{lem:additivity} we may assume $V$ is locally closed. 

\textbf{Second reduction step:} Write $V = Z_1 \setminus Z_2$ with $Z_2 \sub Z_1$ closed, then by \ref{lem:additivity} we have
\begin{equation*}
\iota^{naive}_V(\ca D,  \ca E) = \iota^{naive}_{Z_1}(\ca D,  \ca E) - \iota^{naive}_{Z_2}(\ca D,  \ca E), 
\end{equation*}
So we may assume $V$ is closed. 

\textbf{Third reduction step:} Write $V = Z(f_1, \dots, f_r)$, with $f_i \in R$. For a
subset $T \sub \{1, \dots, r\}$ define $S_T = \on{Spec} \left((\prod_{i \in T} f_i)^{-1} R\right)$. Then by inclusion-exclusion we have 
\begin{equation*}
\iota^{naive}_V(\ca D,  \ca E) = \sum_{T \sub \{1, \dots, r\}} (-1)^{\# T}
  \iota^{naive}_{S_T} (\ca D,  \ca E). 
\end{equation*}
Since $S_T$ is affine, we are reduced to the case where $V$ is the whole of $\ca C = \on{Spec} R$. 

\textbf{Concluding the algorithm:} 
Since forming quotients commutes with flat base-change, we obtain
\begin{equation*}
\iota_{\ca C}^{naive}(\ca D,  \ca E) = \on{length}_R \left( \frac{R\otimes_{\bb Z} \bb Z_\fp}{I_{\ca D} \otimes_{\bb Z} \bb Z_\fp + I_{\ca E} \otimes_{\bb Z} \bb Z_\fp} \right)\log \#k(\fp). 
\end{equation*}
This can be computed using \cite[Algorithm 1]{Mul14}. For efficiency we compute this length working modulo a sufficiently large power of $\fp$, which will be determined in \ref{rem:precision}. 

\begin{remark}
  Note that the third reduction step is exponential in $r$. In the examples we've
  computed, the largest value of $r$ was~4.
\end{remark}

\subsection{Computing the intersection pairing}\label{sec:intersection_pairing}

Let $C/\bb Q_\fp$ be a smooth projective curve, $\ca C/\bb Z_\fp$ a regular model, and $\ca
D$, $\ca E$ two divisors on $\ca C$ without common component. In this section, we describe several approaches to computing the intersection pairing $\iota(\ca D, \ca E)$, depending on how $\ca C$ is given to us. 

\subsubsection*{Regular model given by affine charts and glueing data}
Suppose that the regular model $\ca C$ is given as a list of affine charts $C_1, \dots, C_n$ and glueing data. We partition $\ca C$ into constructible subsets $V_i$ by, for each $i \in \{ 1, \dots, n\}$, setting $V_i = C_i\setminus (\cup_{j < i} C_j)$. Then the intersection pairing is given by 
\begin{equation*}
\iota(\ca D, \ca E) = \sum_{i \in \{ 1, \dots, n\}} \iota^{naive}_{V_i}(\ca D, \ca E). 
\end{equation*}

\subsubsection*{Regular model as described by {\tt Magma}}

{\tt Magma}'s regular models implementation (due to Steve Donnelly) describes the model $\ca C$ in a slightly different way. It constructs a regular model by repeatedly blowing up non-regular points and/or components in a proper model. In this way, it creates a list of affine patches $U_i$ together with open immersions from the generic fibre of the $U_i$ to $C$. For each $i$, it stores a constructible subset $V_i \sub U_i$, consisting of all regular points in the special fibre which did not appear in any of the previous affine patches. These $V_i$ form a constructible partition of the special fibre of a regular model. In this case, we simply compute 
\begin{equation*}
\iota(\ca D, \ca E) = \sum_{i \in \{ 1, \dots, n\}} \iota^{naive}_{V_i}(\ca D, \ca E). 
\end{equation*}

\subsection{Computing the non-archimedean local N\'eron pairing}

Let $C/\bb Q_\fp$ be a smooth projective curve, $\ca C/\bb Z_\fp$ a regular model, $D$ and $E$ degree $0$ divisors on $C$ with disjoint support. In this section we will describe how to compute the local N\'eron pairing $\Span{D, E}_\fp$. 

First we compute the extensions of $D$ and $E$ to horizontal divisors $\ca D$ and $\ca E$ on $\ca C$. We break $D$ and $E$ into their effective and anti-effective parts, then choose some extensions of these ideals to $\ca C$ (the associated subschemes may contain many vertical components). We then saturate these ideals with respect to the prime $\fp$ to obtain (ideals for) horizontal divisors. This works by the following well-known lemma. 

\begin{lemma}\label{lem:saturation}
Let $R$ be a $\bb Z$-algebra, and $I$ an ideal of $R$. The ideal sheaf of the schematic image of $ \on{Spec} R[1/\fp]/ (I \otimes_R R[1/\fp])$ in $ \on{Spec} R$ is given by the saturation 
\begin{equation*}
(I:\fp^\infty) = \{r \in R : \exists n : \fp^n r \in I\}. 
\end{equation*}
\end{lemma}
\begin{proof}
It is immediate that $(I:\fp^\infty)\otimes_R R[1/\fp] = I\otimes_R R[1/\fp]$. We need to check that, for any ideal $J \triangleleft R$ with $J\otimes_R R[1/\fp] = I\otimes_R R[1/\fp]$, we have $J \subseteq (I:\fp^\infty)$. Indeed, if $j \in J$ then we can write ${j}$ as a finite sum of elements $\frac{i}{\fp^{n_i}}$ with $i \in I$, $n_i \in \bb N$, so $\fp^{\max_i n_i}j \in I$, as required. 
\end{proof}

To compute the vertical correction term $\Phi(D)$, we use the algorithm from
\ref{sec:intersection_pairing} to compute the intersection of $\ca D$ with every component
of the fibre of $\ca C$ over $\fp$, then apply simple linear algebra as in \cite[\S
4.5]{Mul14} to find the coefficients of $\Phi(D)$. 

Finally, we use again the algorithm in \ref{sec:intersection_pairing} to compute
\begin{equation*}
\Span{D, E}_\fp = \iota(\ca D +\Phi(D),\ca E +\Phi(E)) = \iota(\ca D, \ca E) + \iota(\Phi(D),\ca E). 
\end{equation*}

\section{The archimedean N\'eron pairing}\label{sec:arch}
\subsection{Green's functions; definition of the pairing}
Let $C/\bb C$ be a smooth projective connected curve of genus $g$, and $\phi$ be a volume form on $C$. If $E$ is a divisor on $C$, we write 
\begin{equation*}
  g_{E,\phi}\colon C(\C) \setminus \supp(E) \to \bb R
\end{equation*}
for a Green's function on $C(\C)$ with respect to $E$ (see \cite[II, \S 1]{Lan88}). If $E$ has degree $0$, and $\phi'$ is another volume form, then $g_{E, \phi} - g_{E, \phi'}$ is constant. If $D = \sum_P n_P P$ is another divisor of degree $0$ with support disjoint from $E$, then the \emph{local N\'eron pairing} is given by 
\begin{equation*}
\Span{D,E}_{\infty} \coloneqq \sum_P n_Pg_{E,\phi}(P);
\end{equation*}
this pairing is bilinear and symmetric, and is independent of the choice of
$\phi$, see~\cite[Theorem~III.5.3]{Lan88}.
As we evaluate $g_{E,\phi}$ in a divisor of degree $0$, we can replace $g_{E,\phi}$ by
$g_{E,\phi}+c$ for a constant $c \in \R$ without changing $\Span{D,E}_{\infty}$.

\subsection{Theta functions; a formula for the pairing}\label{sec:theta}
Let $\{\omega_1,\ldots,\omega_g\}$ be an orthonormal basis of $H^0(C, \Omega^1)$ with
respect to the scalar product $(\omega,\eta)\mapsto\frac{i}{2}\int_{C(\C)}
\omega\wedge \bar{\eta}$ and
let $\varphi \coloneqq \frac{i}{2g}(\omega_1\wedge\bar{\omega_1}+\ldots+\omega_g\wedge\bar{\omega_g})$
be the canonical volume form. We fix  a base point $P_0 \in C(\C)$ and denote by $\alpha:C(\C) \to
J(\C)$ the Abel-Jacobi map
with respect to $P_0$. By abuse of notation, we also denote the additive extension of $\alpha$
to divisors on $C$ by $\alpha$. Following Hriljac, we construct a Green's function by
pulling back the logarithm of a translate of the Riemann theta function $\theta$ along
$\alpha$. Let  $\tau\in \C^{g\times g}$ be the small period matrix of $J(\C)$; it has
symmetric positive definite imaginary part and satisfies $J(\C) \cong \C^g/(\Z^g+\tau\Z^g)$. We 
define 
\[\xymatrix{
  j:\C^g\ar@{->>}[r]&\C^g/(\Z^g+\tau\Z^g)\ar[r]^{\qquad \simeq} &J(\C),}
\]
  Let $\Theta$ denote the theta divisor on $J$ corresponding to $\alpha$.
 By a theorem of
  Riemann (see \cite[Theorem~13.4.1]{Lan83}), there exists a divisor $W$ on $C$ such that
  $2W$ is canonical and such that the translate
  $\Theta_{-{\alpha(W)}}$ of $\Theta$ by $-\alpha(W)$ is the
  divisor of the normalised (in the notation of~\cite[\S13.1]{Lan83}) version of the Riemann theta function
\begin{equation}\label{eq:norm_theta}
 F_{\Theta_{-\alpha(W)}}(z) \coloneqq \theta(z,\tau)\exp\left(\frac{\pi}{2}z^T(\Im
  \tau)^{-1}z\right).
\end{equation}
This $W$ is in fact unique up to linear equivalence, by \cite[Chapter II, theorem 3.10]{Tata1}. 

For the remainder of this section, we suppose that $E=E_1-E_2$, where $E_1$ and $E_2$ are
{\em non-special}. This means that they are effective of degree $g$ with $h^0(C, \ca O(E_i)) = 1$. Because of the bilinearity of the N\'eron pairing, the
following gives a formula to compute $\Span{D,E}_{\infty}$ for all $D\in \Div^0(C)$ with
support disjoint from $E$.

\begin{proposition}\label{P:gfformula}
  Suppose that $D=P_1-P_2$ with $P_1,\, P_2 \in C(\C)$, not in the support of $E$. Then
  \[
    \Span{D,E}_{\infty} = -\log\left|\frac{
      \theta(z_{11},\tau)\cdot\theta(z_{22},\tau)}{\theta(z_{12},\tau)\cdot\theta(z_{21},\tau)}\right|
    -2\pi\Im (z_E)^T\Im (\tau)^{-1}\Im (z_D)
  \]
where $z_D, z_E, z_{ij} \in \C^g$ satisfy $j(z_D) = \alpha(D)$, $j(z_E) = \alpha(E)$ and
  $j(z_{ij}) = \alpha(P_i-E_j+W)$.
\end{proposition}

For the proof of~\ref{P:gfformula} we need the notion of a N\'eron function on
  $J(\C)$, see~\cite[\S13.1]{Lan83}. For each divisor $A \in \Div(J)$, there is {\em a N\'eron function with respect to $A$}, which is uniquely determined up to adding a constant. This is a continuous function $\lambda_A : J(\C) \setminus \supp(A) \to \R$, and together they have the following properties:
  \begin{enumerate}[label=(NF{\arabic*})]
    \item if $A,B \in \Div(J)$, then $\lambda_{A+B} - \lambda_{A} - \lambda_{B}$ is
      constant;
    \item if $f \in \C(J)$, then $\lambda_{\div(f)} +\log|f| $ is constant;
    \item if $A \in \Div(J)$ and $Q \in J(\C)$,
      then $P \mapsto\lambda_{A_Q}(P) - \lambda_{A}(P-Q)$ is constant.
  \end{enumerate}

  A result of N\'eron lets us
 express the N\'eron function of a divisor in terms of
  the  normalised theta function associated to that divisor. In particular, we find:
\begin{lemma}\label{lem:nf1} 
  We get a N\'eron function with respect to ${\Theta_{-\alpha(W)}}$ by mapping $P\in
  J(\C)$ to
  \[
    \lambda_{\Theta_{-\alpha(W)}}(P)
  \colonequals -\log|\theta(z, \tau)|+ \pi \Im(z) ^T \Im(\tau)^{-1} \Im(z),
\]
where $z\in\C^g$ is such that $j(z) = P$.
\end{lemma}
\begin{proof}
Let $H$ denote the Hermitian form with
  matrix $\Im(\tau)^{-1}$; by~\cite[Proposition 13.3.1]{Lan83} this is the Hermitian form (in the language
  of~\cite[\S13.1]{Lan83}) of the divisor ${\Theta_{-\alpha(W)}}$. Because of N\'eron's
  theorem (see~\cite[Theorem~13.1.1]{Lan83})  and because 
  of~\ref{eq:norm_theta}, we get a N\'eron function by mapping $P \in J(\C)$ to
\begin{align*}
  \lambda_{\Theta_{-\alpha(W)}}(P)& \colonequals -\log|F_{\Theta_{-\alpha(W)}}(z)|
  +\frac{\pi}{2}H(z,z)  \\   
  & =-\log| \theta(z,\tau)|- \log \left|\exp\left(\frac{\pi}{2}z^T(\Im
  \tau)^{-1}z\right)\right| + \frac{\pi}{2} z^T\Im(\tau)^{-1}\bar{z}\\
  &= -\log| \theta(z,\tau)|- \frac{\pi}{2}\left(\Re(z)^T(\Im \tau)^{-1}\Re(z) - \Im(z)^T(\Im
  \tau)^{-1}\Im(z)\right) + \frac{\pi}{2} z^T\Im(\tau)^{-1}\bar{z}\\
  &=-\log|\theta(z, \tau)|+ \pi \Im(z) ^T \Im(\tau)^{-1} \Im(z),
 \end{align*}
where $z\in\C^g$ is such that $j(z) = P$.
\end{proof}

\begin{proof}[Proof of~\ref{P:gfformula}]
Let $\Theta^{-} = [-1]^*\Theta$.
  We first find a N\'eron function for $\Theta^{-}_{\alpha(E_j)}$, where $j \in\{1,2\}$.
  Since we have
  \[\Theta^- = \Theta_{-\alpha(2W)}\] by \cite[Theorem~5.5.8]{Lan83},
  property (NF3) implies that 
  \begin{equation}\label{E:neron_Theta}
   \lambda_{j}(P) \coloneqq \lambda_{\Theta_{-\alpha(W)}}(P - \alpha(E_j) + \alpha(W))
  \end{equation}
  is a N\'eron function with respect to $\Theta^{-}_{\alpha(E_j)}$, where
  $\lambda_{\Theta_{-\alpha(W)}}$ is as in~\ref{lem:nf1}.

  Since $E_j$ is non-special, a result of Hriljac (see~\cite[Theorem~13.5.2]{Lan83})
  implies that
  \begin{equation}\label{E:green_neron}
    g_{E_j,\varphi} = \lambda_{j}\circ\alpha + c_j
  \end{equation}
  for some constant $c_j\in \R$. 

  Using~\ref{E:green_neron}, \ref{E:neron_Theta} and~\ref{lem:nf1}, we conclude that
  \begin{align*}
    g_{E_j,\varphi}(P_i) 
    &= \lambda_{\Theta_{-\alpha(W)}}(\alpha(P_i) - \alpha(E_j) +\alpha(W))+c_j \\
    &=  -\log|\theta(z_{ij},\tau)|+\pi \Im(z_{ij}) ^T \Im(\tau)^{-1} \Im(z_{ij})+c_j.
  \end{align*}
  The result now follows from
  \begin{equation*}\label{}
    g_{E,\varphi}(D) = g_{E_1,\varphi}(P_1) - g_{E_2,\varphi}(P_1) - g_{E_1,\varphi}(P_2) + g_{E_2,\varphi}(P_2).
  \end{equation*}
  and the definition of the local N\'eron pairing.
 \end{proof}

\begin{remark}\label{R:steffen_messed_up}
  In~\cite[Corollary~4.16]{Mul14} and~\cite[\S7.3]{Hol12}
  equivalent formulas for $\Span{D,E}_\infty$ were given for the special case of hyperelliptic curves.
  Our \ref{P:gfformula} implies those results, if we use a Weierstrass point as the base
  point for the Abel-Jacobi map; in this case $\alpha(W)=0$. Note that \cite[Corollary~4.16]{Mul14}
  is stated without the assumption that the curve is hyperelliptic, but is false in
  general. We have adapted and corrected the proof given there. Alternatively, one could
  also generalise the proof in~\cite[\S7]{Hol12}.
\end{remark}

\begin{remark}\label{R:special_is_bad}
  In the proof of~\ref{P:gfformula} the condition that $E_1$ and $E_2$ are non-special is
  only used to apply Hriljac's theorem which constructs the Green's function on $C$ by
  pulling back a N\'eron function on $J$ along the Abel-Jacobi map. If the divisor $E_j$
  is non-special, then the intersection of the translate of $\Theta^{-}$ by $\alpha(E_j)$ with the
  curve $C$ recovers the divisor $E_j$ (see \cite[Theorem~5.5.8]{Lan83}), hence we can pull back a
  N\'eron function with respect to $\Theta^{-}_{\alpha(E_j)}$ to obtain a Green's function for the divisor $E_j$. In contrast, if the
  divisor $E_j$ is special then this intersection can (set-theoretically) be much larger,
  so pulling back a N\'eron function does not give anything meaningful.
  Indeed, we have found examples where \ref{P:gfformula} is false
  for special $E_1$ and $E_2$.
\end{remark}

\subsection{Computing the archimedean local N\'eron pairing}
To compute $\Span{D,E}_{\infty}$,  we use the {\tt Magma} code
written by Christian Neurohr for the computation of the small period matrix $\tau$ associated to
$C(\C)$ and the Abel-Jacobi map $\alpha$. See Neurohr's thesis~\cite{NeurohrPhD} for a
description of the algorithm.
This code makes it possible to numerically
approximate these objects efficiently to any desired precision. If $C$ is superelliptic,
then we instead use Neurohr's implementation of the specialised algorithms of
Molin-Neurohr~\cite{MN19} (\url{https://github.com/pascalmolin/hcperiods}). The code requires as input a (possibly singular) plane
model of $C$; this is easy to produce in practice, for instance via projection or
by computing a primitive element of the function field of $C$.

The Riemann theta function can be computed using code already contained in {\tt Magma}.
It is also necessary to find the divisor $W$ in~\ref{P:gfformula}. We first compute 
a canonical divisor and its image under $\alpha$. Then we 
run through all preimages under multiplication by~2 in $\C^g/(\Z^g\oplus\tau\Z^g)$ until we
find the correct $W$ so that $\Theta_{-\alpha(W)}$ is the divisor of the
normalised Riemann theta function, see \ref{sec:tors}.
Once we have the correct $\alpha(W)$, we can compute $\Span{D,E}_\infty$ easily via~\ref{P:gfformula}. 

\begin{remark}
The implementation of Molin-Neurohr and the computation of theta functions in Magma are
rigorous, which means that for superelliptic curves our algorithm returns a provably correct result to any desired
precision, if we disregard possible precision loss. To handle the latter, one would have to
  use interval or ball arithmetic, as implemented, for instance, in {\tt Arb}~\cite{Joh17}. Indeed, Molin and
  Neurohr have 
implemented their algorithms in {\tt Arb}, but we have not
attempted to use this.
In contrast, Neurohr's {\tt Magma}-implementation of his algorithms for
more general curves does not currently yield provably correct output, see the discussion
in~\cite[Section~4.10]{NeurohrPhD}.
\end{remark}

\section{The global height pairing}\label{sec:global}

\subsection{Faltings-Hriljac}\label{S:Gross}
Let $K$ be a global field and let $C /K$ be a smooth, projective, geometrically connected curve of genus
$g>0$ with Jacobian $J = \on{Pic}^0_{C/K}$, and let $D$ and $E$ be degree $0$ divisors on
$C$ with disjoint support. If $v\in M_K$ is a place of $K$, then according to~\cite[III,
\S 5]{Lan88}, the local N\'eron pairing at $v$  satisfies 
\[
  \langle D, \div(f)\rangle_v = -\log|f(D)|_v,
\]
for all rational functions $f \in K(C)^\times$ and  divisors $D\in \Div(C)$ of degree $0$,
with support disjoint from $\div(f)$. Here the absolute values are normalised to satisfy
the product formula and we define $f(D)=\prod_jf(Q_j)^{m_j}$ if $D = \sum m_jQ_j$.
Hence the global N\'eron pairing $\sum_{v \in M_K} \Span{D,E}_v$ does respect linear
equivalence and extends to a symmetric bilinear pairing on the rational points of $J$.

We now relate the global N\'eron pairing to N\'eron-Tate heights.
Write $T$ for the image of $C^{g-1}$ in $\on{Pic}^{g-1}_{C/K}$. Choose a class
$w \in \on{Pic}^{g-1}_{C/K}(\bar{K})$ with $2w$ equal to the canonical class of
$C$ in $\on{Pic}^{2g-2}_{C/K}(K)$. Then the class $\vartheta$ of $T_{-w}$ is a
symmetric ample divisor class on $J_{\bar K}$, and $2\vartheta$ is independent of the choice
of $w$ and is defined over $K$. The following theorem is due to Faltings and Hriljac
\cite{Fal84, Hri85, Gro84}.

\begin{theorem}\label{T:FH}
Let $D$ and $E$ be degree $0$ divisors on $C$ with disjoint support, then
  \[
  \hat h_{2 \vartheta}([D],[E]) = -\sum_{v \in M_K} \Span{D,E}_v. 
  \]
\end{theorem}
In the following, we assume $K=\Q$ for simplicity. We also assume that every element of $J(\Q)$
can be represented using a $\Q$-rational divisor; this always holds if $C$ has a $\Q_v$
rational divisor of degree 1 for all places $v$ of $\Q$,
see~\cite[Proposition~3.3]{PS97}. This assumption is  convenient, as it allows us to
compute the non-archimedean N\'eron pairings over $\Z_p$. If such representatives do not
exist, we could work over finite extensions.

\begin{remark}
  There is a similar decomposition of the $p$-adic height on $J$ due to
  Coleman-Gross~\cite{CG89}, where the local summand at a non-archimedean prime $v \ne p$
  is the N\'eron pairing at $v$, up to a constant factor, and there is no archimedean
  summand.
  Therefore we only need to combine \ref{alg:NA} with an algorithm to compute the summand
  at $p$, which is defined in terms of Coleman integrals, to get a method for the
  computation of  the $p$-adic height on $J$. This would be interesting, for instance, in the context of
  quadratic Chabauty, see the discussion in \cite[\S1.7]{BDMTV}. For hyperelliptic curves,
  such an algorithm is due to Balakrishnan-Besser~\cite{BB12}.
\end{remark}

\subsection{Finding suitable representatives}
Suppose we are given two points $P$, $Q \in J(\bb Q)$, given by $\Q$-rational degree~$0$ divisors $D$
(resp. $E$) representing $P$ (resp. $Q$), and wish to compute the height
pairing $\hat h_{2 \vartheta}(P,Q)$. 
The local N\'eron pairings are only defined for divisors with disjoint support. 
If $D$ and $E$ have common support, we can move $E$ away from $D$ using strong
approximation, see~\cite[\S4.9.4]{NeurohrPhD}. This algorithm computes a rational function
$f_P$ for $P$ in the common support of both $D$ and $E$ such that $v_P(\div(f_P)) = -1$
and such that $\supp(\div(f_P)) \cap \supp(D) = \{P\}$. We replace $E$ by
$E+\sum_Pv_P(E)\div(f_P)$. 

In practice, the following approach is often simpler: 
reduce multiples of $E$ along a suitable divisor until this yields a divisor $E'$ with support
disjoint from $D$. Due to the bilinearity of the N\'eron
pairings, we can  replace $E$ by $E'$, see also~\cite[\S4.1]{Mul14}.
In both approaches, the bottleneck is the computation of Riemann-Roch spaces~\cite{Hes02}. 
We can also use these methods to ensure that $E$ can be written as the difference of non-special divisors.

\subsection{Identifying relevant primes}
Fix degree~$0$ divisors $D$ and $E$ with disjoint support. A-priori the expression in
\ref{T:FH} is an infinite sum; we must identify a finite set $R$ of `relevant' places
outside which we can guarantee that the local N\'eron pairing of $D$ and $E$ vanishes.
This set $R$ will be the union of three sets; the infinite place, the primes where $C$ has
bad reduction, and another finite set containing the other primes at which $D$ and $E$ meet. 

\subsubsection{Bad primes}
We assume that $C$ is given with an embedding $i\colon C \to \bb P^n_{\bb Q}$ in some
projective space, and we write $\bar C$ for some proper model of $C$ inside $\bb P^n_{\bb
Z}$. For instance, we could always take $n=3$ in practice. The standard affine charts of $\bb P^n_{\bb Z}$ induce an affine cover of $\bar C$,
and we check non-smoothness of $\bar C$ on each chart of the cover separately. Suppose
that a chart of $\bar C$ is given by an ideal $I \triangleleft \bb Z[x_1, \dots, x_n]$,
and $I$ is generated by $f_1, \dots, f_r$. Then a Gr\"obner basis for the jacobian ideal of
$I$ will contain exactly one integer, and its prime factors are exactly those primes over which this affine patch fails to be smooth over $\bb Z$.

\subsubsection{Primes where $D$ and $E$ may meet}\label{sec:primes_meeting}
We reduce to the case where $D$ and $E$ are effective. Then we proceed as above, embedding
$C$ in some projective space, and taking some model $\bar C$. On each affine chart, we
take some proper models $\bar D$ and $\bar E$ of $D$ and $E$. If $\bar C$ is cut out by
$I$, and $\bar D$ and $\bar E$ by ideals $I_D$ and $I_E$, then a Gr\"obner basis for $I +
I_D + I_E$ has exactly one entry that is an integer (we denote it $n_{D,E}$), and again the prime factors of $n_{D,E}$ contain all the primes above which $\bar D$ and $\bar E$ meet. 

 \begin{remark}\label{rem:precision}
The final step in \ref{alg:NA} computes lengths of modules over $\bb Z$. In fact, it is
   much more efficient to work modulo a large power of the prime $\fp$. The techniques
   just described to identify a finite set of relevant primes can also be used to bound
   the required precision. If either of the divisors concerned is supported on the special
   fibre, then it suffices to work modulo $p^n$ where $n$ is the maximum of the
   multiplicities of the components. If both divisors $D$ and $E$ are horizontal, then the
   maximal power of the prime $\fp$ dividing the integer $n_{D,E}$ (defined just above) is an upper bound on the intersection number, and so provides a sufficient amount of $p$-adic precision. Note that resolving singularities by blowing up can only decrease the naive intersection multiplicity, and so this bound is also valid at bad places, as long as the regular model we use is obtained by blowing up $\bar C$. 
\end{remark}

\begin{remark}
The integer $n_{D,E}$ can become very large, even if the equations for $C$, $D$ and $E$ have small coefficients (moving $E$ by linear equivalence often makes the coefficients very much larger). As such, factoring it can become a bottleneck. In principle this factorisation should be avoidable; for example, one can treat the bad primes separately, then one has a global regular model over the remaining primes and the multiplicity can be computed there directly. Algorithms for computing heights on genus 1 and 2 curves without factorisation can be found in \cite{MS16a,MS16b}. 
\end{remark}

\section{Examples}\label{sec:examples}
We have implemented our algorithm in {\tt Magma}. Besides testing it against the code in
{\tt Magma} (based on~\cite{FS97, Sto02, Mul14}) for some hyperelliptic Jacobians, 
we also tested it on a few Jacobians of smooth plane quartics, though the algorithm is by no means limited
to genus 3. At present we can only compute the regulator up to an integral square, because our algorithm only lets us {\em
compute} the N\'eron-Tate height -- we cannot use it to {\em enumerate} points of
bounded N\'eron-Tate height, which would be required for provably determining generators of
$J(\Q)$ with the usual saturation techniques, see the introduction and \cite{Sik95, Sto02}.
If $C$ is hyperelliptic of genus at most 3, then this is possible using the algorithms
discussed in the introduction.
For an Arakelov-theoretic approach to this problem see \cite{Hol14}.

\subsection{A torsion example}\label{sec:tors}

Let $C \colon X^3Y - X^2Y^2 - X^2Z^2 - XY^2Z + XZ^3 + Y^3Z = 0$ in $\mathbb{P}^2_{\Q}$ from \cite[Example~12.9.1]{BPS16}. Its Jacobian is of rank 0 and has 51 rational torsion points. Its bad primes are 29 and 163, but the model over $\Z_{29}$ and $\Z_{163}$ given by the same equation is already regular.

Let $D = D_1 - D_2$ and $E = 3 \cdot E_1 - 3 \cdot E_2$, where $D_1 = (1:0:1)$, $D_2 =
(1:1:0)$, $E_1 = (1:0:0)$ and $E_2 = (1:1:1)$. 
We choose this $E$ rather than $E_1-E_2$ because of the conditions imposed on $E$
in~\ref{sec:theta}.
Then the computations for the intersections
can be done on the affine patch where $X \neq 0$ of $C$. Consider the ring $$R = \Z[y,z] /
(y - y^2 - z^2 - y^2z + z^3 + y^3z),$$ which is regular. The ideals $I_{D_1} = (y, z-1)$
and $I_{3 \cdot E_1} = (y^3, z^3)$ are coprime in $R$, and hence there will be no
intersection between $D_1$ and $E_1$ at any of the non-archimedean places. In the same
way, there is no non-archimedean intersection between $D_1$ and $E_2$, between $D_2$ and
$E_1$, and between $D_2$ and $E_2$. Note that also $\Phi(D)$ and $\Phi(E)$ can be taken to be 0, as the special fibres of the regular models we computed are irreducible. 

For the computation of the archimedean contribution, we first need a canonical divisor which, for practical reasons, has to be supported outside infinity (i.e.\ $X = 0$). For this purpose, we pick $K = \mathop\mathrm{div} ( (z-1)^2 / (y^2z^2) \, dz ).$

Then we use Neurohr's algorithm~\cite{NeurohrPhD} to compute the small period matrix
$\tau$, and $\alpha(D_1), \alpha(D_2)$, $\alpha(E_1), \alpha(E_2),$ and $\alpha(K)$, where $\alpha \colon C(\C) \to J(\C)$ is the embedding whose base point is chosen by Neurohr's algorithm, which turned out to be the point $(1 : -2 : -2.6615...)$ in this case. To find the appropriate divisor $W$ with $2 W = K$ out of the $2^6 = 64$ candidates, we try the 64 candidates for $\alpha(W)$ and compute for which one the function $\theta(z, \tau)$ has a pole at a point $z \in \C^g$ satisfying $j(z) = \alpha(D_1) + \alpha(D_2) - \alpha(W)$ (which is in $\Theta$).
Then we finally compute the expression in \ref{P:gfformula}, and find that the archimedean contribution is approximately 0, or to be more precise, the result was approximately $2 \cdot 10^{-29}$ when computing with 30 decimal digits of precision.

\subsection{An example in rank 1}

Let $C$ be the smooth plane quartic curve over $\Q$ given by $$X^2 Y^2 - X Y^3 - X^3 Z - 2 X^2 Z^2 + Y^2 Z^2 - X Z^3 + Y Z^3 = 0.$$ This is the curve from \cite[Example~12.9.2]{BPS16}. It has rank 1 and trivial rational torsion subgroup. Its bad primes are 41 and 347, but the model over $\Z_{41}$ and $\Z_{347}$ given by the same equation is already regular.

Let $D = D_1 - D_2$ and $E = 3 \cdot E_1 - 3 \cdot E_2$, where $D_1 = (1:0:-1)$, $D_2 =
(1:1:-1)$, $E_1 = (1:1:0)$ and $E_2 = (1:4:-3)$. The computations for the intersections
can be done on the affine patch of $C$ where $X \neq 0$. Consider the ring $$R = \Z[y,z] / (y^2 - y^3 - z - 2z^2 - y^2z^2 - z^3 - yz^3).$$ The sum of the two ideals $I_{D_1} = (y, z+1)$ and $I_{E_2} = (y - 4, z + 3)$ inside $R$ is $(2,y,z+1)$. Hence, the only place where $D_1$ and $E_2$ could possibly intersect is the prime 2.
At 2, the length of $\Z_{(2)}[y,z] / (2, y, z+1) \cong \F_2$ as $R_{(2)}$-module is 1, so $\iota(D_1, E_2) = \log(2)$. There is no intersection between $D_1$ and $E_1$, between $D_2$ and $E_1$, and between $D_2$ and $E_2$. Moreover, $\Phi(D)$ and $\Phi(E)$ can be taken to be 0 again. Hence, the intersection pairing $\langle D, E \rangle_{\mathfrak{p}}$ equals $-3\log(2)$ if $\mathfrak{p} = (2)$, and 0 otherwise.

We computed the archimedean contribution in the same way as in the previous example, and
we found it to be $-0.013563$. Hence, the N\'eron-Tate height pairing is $\hat{h}_{2\vartheta}([D], [E]) =  2.0930$.

We performed an analogous computation for the points $F = (0:1:0) - D_2$, and $G = 3\cdot
E_2 - 3 \cdot (0:1:-1)$, and found that $\hat{h}_{2\vartheta}([F], [G]) = -0.59966$. We
computed this with 30 decimal digits of precision, and found numerically that $-414 \cdot
\hat{h}_{2\vartheta}([D],[E]) = 1445 \cdot \hat{h}_{2\vartheta}([F], [G])$. We deduced that $g = [E] - [F]$ is a possible generator for the Mordell-Weil group, and the relation between the heights suggested the relations $[D] = 17 \cdot g$ , $[E] = 255 \cdot g$, $[F] = -69 \cdot g$, and $[G] = 18 \cdot g$, which we confirmed in the Mordell-Weil group. If $g$ is indeed the generator of the Mordell-Weil group, then the regulator is $0.00048282$.

\subsection{The split Cartan modular curve of level 13}
Let $C$ denote the smooth plane quartic curve given by the equation
\begin{equation}\label{E:cartan_model}
  (-Y-Z)X^3 +(2Y^2 +YZ)X^2 +(-Y^3 +Y^2Z -2YZ^2 +Z^3)X+(2Y^2 Z^2 -3YZ^3 )=0.
\end{equation}
By work of Baran~\cite{Bar14a, Bar14b} this curve is isomorphic 
to the modular curve $X_{s}(13)$ which classifies elliptic curves whose
Galois representation is contained in a normaliser of a split Cartan subgroup of
$\GL_2(\F_{13})$, as well as its non-split counterpart $X_{ns}(13)$. 
Assuming the Generalised Riemann Hypothesis, Bruin-Poonen-Stoll~\cite[Example~12.9.3]{BPS16} prove that $J(\Q)$ has
rank~3; an unconditional proof is given in~\cite{BDMTV}. By a result of Balakrishnan, Dogra, Tuitman, Vonk and the third-named author~\cite{BDMTV}, there are
precisely 7 rational points on $C$. Using reduction modulo small primes,
Bruin-Poonen-Stoll show that the points
\[
  P_0 \coloneqq (1:0:0),\, P_1 \coloneqq (0:1:0), \, P_2 \coloneqq (0:0:1), \, P_3
  \coloneqq (-1:0:1) \in C(\Q)
\]
have the property that
\[
  [P_1-P_0], [P_2-P_0], [P_3 - P_0] 
\]
on the Jacobian $J$ of $C$ generate a subgroup $G$ of $J(\Q)$ of rank 3, which contains all differences of rational
points.  Therefore the regulator of $J/\Q$ differs from the regulator of $G$ multiplicatively by an integral
square.

The height pairings that we obtain by using our code are:
\begin{center}
\begin{tabular}{|c|ccc|} \hline
	&$[P_1-P_0]$	&$[P_2-P_0]$	&$[P_3-P_0]$	\\ \hline
$[P_1-P_0]$	&0.78401	&0.59540	&0.32516	\\
$[P_2-P_0]$	&0.59540	&0.98372	&0.37437	\\
$[P_3-P_0]$	&0.32516	&0.37437	&0.18861	\\ \hline
\end{tabular}
\end{center}

Hence, the regulator is $9.6703 \cdot 10^{-3}$ up to an integral square factor.

The work of Gross-Zagier~\cite{GZ86} and Kolyvagin-Logachev~\cite{KL89} implies that the rank part of BSD holds in this
example, that the Shafarevich-Tate group is finite, and that the full conjecture of Birch
and Swinnerton-Dyer holds up to an integer. 
We give numerical evidence that it holds up to an integral square.
This is the first non-hyperelliptic example where the BSD invariants (except the order of
the Shafarevich-Tate group) have been computed; for hyperelliptic examples see
\cite{FLSSSW, vB17}.

In \cite[Example~12.9.3]{BPS16}, it is already shown that $J$ has no non-trivial rational torsion. It is verified easily that the model in $\Z$ given by the same equation as in \ref{E:cartan_model} is regular at all primes. Hence, all Tamagawa numbers equal 1.
For the value of the $L$-function, we use that $J$ is isogenous to the abelian variety
$A_f$ associated to a newform $f \in S_2(\Gamma_0(169))$ with Fourier coefficients in
$\Q(\zeta_7)^+$. Hence we have $$L(J,s) = \prod_\sigma L(f^{\sigma}, s),$$ where $\sigma$ runs
through $\Gal(\Q(\zeta_7)^+/\Q)$. Computing the factors on the right hand side using {\tt
Magma}, we obtained $\lim_{s \rightarrow 1} L(J,s) \cdot (s-1)^{-3} \approx 0.76825$.

For the real period, we used the code of Neurohr to compute a big period matrix $\Lambda$
for $J$. One can then apply the methods of the first-named author
\cite[Algorithm~13]{vB17} to check that the differentials used for the computation of
the big period matrix are 3 times a set of generators for the canonical sheaf. Hence, the
real period is $\frac1{27}$ times the covolume of the lattice generated by the 6 columns
of $\Lambda + \overline{\Lambda}$ inside $\R^3$. We computed the real period to be $79.444$ and
checked that this value agrees with the real volume of $A_f$.

Assuming our value for the regulator is correct, the BSD formula predicts that the size of
the Shafarevich-Tate group is $\frac{0.76825}{9.6703 \cdot 10^{-3} \cdot 79.444} \approx
1.0000$, which is consistent with the result of \cite{PS99} proving that the size
of the group is a square in this case, if it is finite.

\subsection{An example with very bad reduction}

In all the examples we tried so far, the naive model over $\bb Z$ happened to be regular. We wanted to try an curve where this was far from the case, but still with Jacobian of positive rank. We searched for a curve with some rational points, and very bad reduction at a small prime, finding the genus 3 curve $C$ over $\Q$ given by $$3x^3y + 5 xy^2z + 5y^4 - 1953125z^4 = 0,$$
 with rational points $P_1 = (1:0:0)$ and $P_2 = (0:25:1)$. The bad primes are $3$, $5$, $17$, $358166959$, $523687087967$. For the three largest prime factors, the na\"ive models are already regular. The special fibre of the regular model produced by {\tt Magma} over the prime $3$ has 4 irreducible components, with multiplicities $[1,1,2,2]$, and intersection matrix

\begin{equation*}
\left[
\begin{matrix}
-6&0&2&1\\
 0 &-2&0&1\\
 2&0&-2&1\\
 1&1&1&-2\\
\end{matrix}
\right]. 
\end{equation*}

That over the prime 5 has 9 components, with multiplicities $[1, 1, 1, 1, 1, 1, 2, 3, 3]$ and intersection matrix
\begin{equation*}
\left[
\begin{matrix}
-1&0&1&0&0&0&0&0&0\\
 0&-4&0&0&0&1&0&0&1\\
 1&0&-2&0&0&1&0&0&0\\
 0&0&0&-3&1&0&1&0&0\\
 0&0&0&1&-2&1&0&0&0\\
 0&1&1&0&1&-3&0&0&0\\
0&0&0&1&0&0&-2&0&1\\
0&0&0&0&0&0&0&-1&1\\
 0&1&0&0&0&0&1&1&-2\\
\end{matrix}
\right]. 
\end{equation*}

We define a degree~$0$ divisor $D = P_1 - P_2$, and compute the height pairing of $D$ with itself, obtaining
\begin{equation*}
\hat{h}_{2\vartheta}(D,D) \approx 3.2107. 
\end{equation*}
In particular, this shows that $D$ is not torsion on the Jacobian, hence the rank is at least 1 (probably, it equals 1) and the regulator is probably 3.2107, though of course there might exist a generator of smaller height. 

The computation took around 5 minutes, with $90\%$ of this time spent on the saturation step (\ref{lem:saturation}). Each saturation carried out took around 1.5 seconds, but the complexity of the reduction types meant that many such steps were necessary.


\begin{bibdiv}
\begin{biblist}

			
\bib{BB12}{article}{
   author={Balakrishnan, Jennifer S.},
   author={Besser, Amnon},
   title={Computing local $p$-adic height pairings on hyperelliptic curves},
   journal={Int. Math. Res. Not. IMRN},
   date={2012},
   number={11},
   pages={2405--2444},
}		



\bib{BDMTV}{misc}{
   author={Balakrishnan, Jennifer S.},
   author={Dogra, Netan},
   author={M\"uller, J. Steffen},
   author={Tuitman, Jan},
   author={Vonk, Jan},
   title={  Explicit Chabauty-Kim for the Split Cartan Modular Curve of Level 13},
   note={Preprint, {\url{https://arxiv.org/abs/1711.05846}}}
}


\bib{Bar14a}{article}{
   author={Baran, Burcu},
   title={An exceptional isomorphism between modular curves of level 13},
   journal={J. Number Theory},
   volume={145},
   date={2014},
   pages={273--300},
}

\bib{Bar14b}{article}{
   author={Baran, Burcu},
   title={An exceptional isomorphism between level 13 modular curves via
   Torelli's theorem},
   journal={Math. Res. Lett.},
   volume={21},
   date={2014},
   number={5},
   pages={919--936},
}

\bib{vB17}{misc}{
  author={van Bommel, Raymond},
   title={ Numerical verification of the Birch and Swinnerton-Dyer conjecture for
   hyperelliptic curves of higher genus over $\mathbb{Q}$ up to squares},
   note={Preprint, {\url{https://arxiv.org/abs/1711.10409}}},
}


\bib{BPS16}{article}{
   author={Bruin, Nils},
   author={Poonen, Bjorn},
   author={Stoll, Michael},
   title={Generalized explicit descent and its application to curves of
   genus 3},
   journal={Forum Math. Sigma},
   volume={4},
   date={2016},
   pages={e6, 80},
}
		


\bib{CG89}{article}{
   author={Coleman, Robert F.},
   author={Gross, Benedict H.},
   title={$p$-adic heights on curves},
   conference={
      title={Algebraic number theory},
   },
   book={
      series={Adv. Stud. Pure Math.},
      volume={17},
      publisher={Academic Press, Boston, MA},
   },
   date={1989},
   pages={73--81},
}
\bib{Dok18}{article}{
   author = {{Dokchitser}, Tim},
    title = {Models of curves over DVRs},
  journal = {ArXiv e-prints},
   eprint = {1807.00025},
     year = {2018}
}

\bib{Fal84}{article}{
   author={Faltings, Gerd},
   title={Calculus on arithmetic surfaces},
   journal={Ann. of Math. (2)},
   volume={119},
   date={1984},
   number={2},
   pages={387--424},
}


	
\bib{FLSSSW}{article}{
   author={Flynn, E. Victor},
   author={Lepr\'evost, Franck},
   author={Schaefer, Edward F.},
   author={Stein, William A.},
   author={Stoll, Michael},
   author={Wetherell, Joseph L.},
   title={Empirical evidence for the Birch and Swinnerton-Dyer conjectures
   for modular Jacobians of genus 2 curves},
   journal={Math. Comp.},
   volume={70},
   date={2001},
   number={236},
   pages={1675--1697},
}


\bib{FS97}{article}{
   author={Flynn, E. V.},
   author={Smart, N. P.},
   title={Canonical heights on the Jacobians of curves of genus $2$ and the
   infinite descent},
   journal={Acta Arith.},
   volume={79},
   date={1997},
   number={4},
   pages={333--352},
}




\bib{Gro84}{article}{
   author={Gross, Benedict H.},
   title={Local heights on curves},
   conference={
      title={Arithmetic geometry},
      address={Storrs, Conn.},
      date={1984},
   },
   book={
      publisher={Springer, New York},
   },
   date={1986},
   pages={327--339},
}


\bib{GZ86}{article}{
   author={Gross, Benedict H.},
   author={Zagier, Don B.},
   title={Heegner points and derivatives of $L$-series},
   journal={Invent. Math.},
   volume={84},
   date={1986},
   number={2},
   pages={225--320},
}
	

\bib{Hes02}{article}{
   author={Hess, F.},
   title={Computing Riemann-Roch spaces in algebraic function fields and
   related topics},
   journal={J. Symbolic Comput.},
   volume={33},
   date={2002},
   number={4},
   pages={425--445},
}

\bib{HS00}{book}{
   author={Hindry, Marc},
   author={Silverman, Joseph H.},
   title={Diophantine geometry},
   series={Graduate Texts in Mathematics},
   volume={201},
   note={An introduction},
   publisher={Springer-Verlag, New York},
   date={2000},
   pages={xiv+558},
   isbn={0-387-98975-7},
   isbn={0-387-98981-1},
   review={\MR{1745599}},
   doi={10.1007/978-1-4612-1210-2},
}
\bib{MN19}{article}{
  author={Molin, Pascal},
  author={Neurohr, Christian},
   title={Computing period matrices and the Abel-Jacobi map of superelliptic curves},
   journal={Math. Comp.},
    date={2019},
   volume={89},
   pages={847--888},
}


\bib{Hol12}{article}{
   author={Holmes, David},
   title={Computing N\'eron-Tate heights of points on hyperelliptic Jacobians},
   journal={J. Number Theory},
   volume={132},
   date={2012},
   number={6},
   pages={1295--1305},
}


\bib{Hol14}{article}{
   author={Holmes, David},
   title={An Arakelov-theoretic approach to na\"\i ve heights on hyperelliptic
   Jacobians},
   journal={New York J. Math.},
   volume={20},
   date={2014},
   pages={927--957},
}

		
\bib{Hri85}{article}{
   author={Hriljac, Paul},
   title={Heights and Arakelov's intersection theory},
   journal={Amer. J. Math.},
   volume={107},
   date={1985},
   number={1},
   pages={23--38},
}		

\bib{Joh17}{article}{
  author = {Johansson, F.},
  title = {Arb: efficient arbitrary-precision midpoint-radius interval arithmetic},
  journal = {IEEE Transactions on Computers},
  year = {2017},
  volume = {66},
  number = {8},
  pages = {1281--1292},
  doi = {10.1109/TC.2017.2690633},
}

\bib{KL89}{article}{
   author={Kolyvagin, V. A.},
   author={Logach\"ev, D. Yu.},
   title={Finiteness of the Shafarevich-Tate group and the group of rational
   points for some modular abelian varieties},
   language={Russian},
   journal={Algebra i Analiz},
   volume={1},
   date={1989},
   number={5},
   pages={171--196},
   translation={
      journal={Leningrad Math. J.},
      volume={1},
      date={1990},
      number={5},
      pages={1229--1253},
   },
}
\bib{Lan88}{book}{
   author={Lang, Serge},
   title={Introduction to Arakelov theory},
   publisher={Springer-Verlag, New York},
   date={1988},
   pages={x+187},
   isbn={0-387-96793-1},
}


\bib{Lan83}{book}{
   author={Lang, Serge},
   title={Fundamentals of Diophantine geometry},
   publisher={Springer-Verlag, New York},
   date={1983},
   pages={xviii+370},
   isbn={0-387-90837-4},
}



\bib{Liu02}{book}{
   author={Liu, Qing},
   title={Algebraic geometry and arithmetic curves},
   series={Oxford Graduate Texts in Mathematics},
   volume={6},
   note={Translated from the French by Reinie Ern\'e;
   Oxford Science Publications},
   publisher={Oxford University Press, Oxford},
   date={2002},
   pages={xvi+576},
   isbn={0-19-850284-2},
}

					



\bib{Mul14}{article}{
   author={M\"uller, J. Steffen},
   title={Computing canonical heights using arithmetic intersection theory},
   journal={Math. Comp.},
   volume={83},
   date={2014},
   number={285},
   pages={311--336},
}



\bib{MS16a}{article}{
   author={M\"uller, J. Steffen},
   author={Stoll, Michael},
   title={Computing canonical heights on elliptic curves in quasi-linear
   time},
   journal={LMS J. Comput. Math.},
   volume={19},
   date={2016},
   number={suppl. A},
   pages={391--405},
}		


\bib{MS16b}{article}{
   author={M\"uller, J. Steffen},
   author={Stoll, Michael},
   title={Canonical heights on genus-2 Jacobians},
   journal={Algebra Number Theory},
   volume={10},
   date={2016},
   number={10},
   pages={2153--2234},
}

\bib{Tata1}{article}{
	Author = {Mumford, David},
	Publisher = {Birkh{\"a}user},
	Title = {{Tata lectures on theta I}},
	date = {1983}
}


\bib{Ner65}{article}{
   author={N\'{e}ron, A.},
   title={Quasi-fonctions et hauteurs sur les vari\'{e}t\'{e}s ab\'{e}liennes},
   language={French},
   journal={Ann. of Math. (2)},
   volume={82},
   date={1965},
   pages={249--331},
   issn={0003-486X},
   review={\MR{0179173}},
   doi={10.2307/1970644},
}
	
\bib{NeurohrPhD}{thesis}{
  author={Neurohr, Christian},
   title={Efficient integration on Riemann surfaces \& applications},
   date={2018},
   organization={Carl von Ossietzky Universit\"at Oldenburg},
   type={PhD thesis},
   note={\url{http://oops.uni-oldenburg.de/3607/1/neueff18.pdf}},
}



\bib{PS97}{article}{
   author={Poonen, Bjorn},
   author={Schaefer, Edward F.},
   title={Explicit descent for Jacobians of cyclic covers of the projective
   line},
   journal={J. Reine Angew. Math.},
   volume={488},
   date={1997},
   pages={141--188},
}



\bib{PS99}{article}{
   author={Poonen, Bjorn},
   author={Stoll, Michael},
   title={The Cassels-Tate pairing on polarized abelian varieties},
   journal={Ann. of Math. (2)},
   volume={150},
   date={1999},
   number={3},
   pages={1109--1149},
}
			

				


\bib{Sik95}{article}{
   author={Siksek, Samir},
   title={Infinite descent on elliptic curves},
   journal={Rocky Mountain J. Math.},
   volume={25},
   date={1995},
   number={4},
   pages={1501--1538},
}



\bib{Sil88}{article}{
   author={Silverman, Joseph H.},
   title={Computing heights on elliptic curves},
   journal={Math. Comp.},
   volume={51},
   date={1988},
   number={183},
   pages={339--358},
}


\bib{Sto02}{article}{
   author={Stoll, Michael},
   title={On the height constant for curves of genus two. II},
   journal={Acta Arith.},
   volume={104},
   date={2002},
   number={2},
   pages={165--182},
}



\bib{Sto17}{article}{
   author={Stoll, Michael},
   title={An explicit theory of heights for hyperelliptic Jacobians of genus
   three},
   conference={
      title={Algorithmic and experimental methods in algebra, geometry, and
      number theory},
   },
   book={
      publisher={Springer, Cham},
   },
   date={2017},
   pages={665--715},
}				


\end{biblist}
\end{bibdiv}
\end{document}